\documentclass[10 pt, onecolumn, oneside, a4paper, reqno]{amsart}
\usepackage{amssymb}
\usepackage{mathrsfs}
\usepackage{mathtools}
\usepackage{amsmath}
\usepackage{graphicx}
\usepackage{pifont}
\usepackage{color}
\usepackage[all]{xy}
\usepackage{type1cm}
\usepackage{hyperref}
\usepackage{color}
\hypersetup{colorlinks=true,
     breaklinks=true,
     linkcolor=blue,
      pageanchor=true}
\newcommand{\xiaowuhao}{\fontsize{9pt}{\baselineskip}\selectfont}

\newcommand{\xiaoliuhao}{\fontsize{7pt}{\baselineskip}\selectfont}

\newtheorem{thm}{Theorem}[section]

\newtheorem{cor}[thm]{Corollary}
\newtheorem{lem}[thm]{Lemma}
\newtheorem{prop}[thm]{Proposition}

\theoremstyle{definition}

\newtheorem{exam}[thm]{Example}

\newtheorem{rem}[thm]{Remark}

\newcommand{\s}{\stackrel}

\newcommand{\C}{\mathcal C}
\newcommand{\D}{\mathcal D}

\newcommand{\K}{\mathcal K}

\renewcommand{\S}{\mathcal S}

\newcommand{\T}{\mathcal T}

\newcommand{\X}{\mathcal X}
\newcommand{\Y}{\mathcal Y}

\DeclareMathOperator*{\add}{\mathsf{add}}

\DeclareMathOperator{\cone}{\mathsf{cone}}

\DeclareMathOperator{\Filt}{\mathsf{Filt}}

\DeclareMathOperator{\Hom}{\mathsf{Hom}}

\DeclareMathOperator*{\Mod}{\mathsf{Mod}\!-}

\DeclareMathOperator*{\op}{\mathsf{op}}
 
 \DeclareMathOperator*{\projA}{\mathsf{proj}-\mathnormal{A}}
  \DeclareMathOperator*{\projB}{\mathsf{proj}-\mathnormal{B}}
   \DeclareMathOperator*{\projC}{\mathsf{proj}-\mathnormal{C}}

\DeclareMathOperator{\RHom}{\mathsf{RHom}}
 \DeclareMathOperator*{\smod}{\mathsf{mod}-}
 \DeclareMathOperator*{\smodA}{\mathsf{mod}-\mathnormal{A}}
 \DeclareMathOperator*{\smodB}{\mathsf{mod}-\mathnormal{B}}
 \DeclareMathOperator*{\smodC}{\mathsf{mod}-\mathnormal{C}}
 \DeclareMathOperator*{\smodeAe}{\mathsf{mod}-\mathnormal{eAe}}
 \DeclareMathOperator*{\smodAeAe}{\mathsf{mod}-\mathnormal{A/AeA}}
  
\title[Gluing simple-minded collections]{Gluing simple-minded collections in triangulated categories}

\author[Y. Sun & Y. Zhang]{Yongliang Sun and Yaohua Zhang*}

\address{\normalfont{Yongliang Sun \\School of Mathematics and Physics, Yancheng Institute of Technology, Jiangsu 224003, People's Republic of China}}
\email{syl13536@126.com}

\address{\normalfont{Yaohua Zhang \\ Hubei Key Laboratory of Applied Mathematics, Faculty of Mathematics and Statistics, Hubei University, Wuhan, 430062, China}}
\email{2160501008@cnu.edu.cn}

\thanks{* Corresponding author.}


\begin{document}
\begin{abstract}
We provide a technique to glue simple-minded collections along a recollement of Hom-finite Krull-Schmidt triangulated categories over a field. This gluing technique for simple-minded collections is shown to be compatible with those for gluing bounded $t$-structures, silting objects, and co-$t$-structures in the literature. Furthermore, it also enjoys the properties of preserving partial order and commuting with the operation of mutation.
\end{abstract}
\maketitle
\setcounter{tocdepth}{1}
\tableofcontents

\section{Introduction}
A simple-minded collection in a triangulated category is a set of objects whose properties mimic the properties of the set of simple modules of a finite-dimensional algebra. It was introduced by AI-Nofayee firstly with a name of cohomologically Schurian set of generators \cite{AI2009}, and later by Keller-Nicol\'{a}s \cite{KN2013}, Koenig-Yang \cite{KY2010, KY2014}, Rickard-Rouquier \cite{RR2017} in different settings. These studies set up a strong relation between simple-minded collections and bounded $t$-structures. Indeed, in a Hom-finite Krull Schmidt triangulated category, there is a one-to-one correspondence between isomorphism classes of simple-minded collections and so-called algebraic $t$-structures, { i.e. bounded $t$-structures with the hearts length categories with finitely many isomorphism classes of simple objects} (see \cite[Subsection 5.3, Proposition 5.4]{KY2014} and \cite[Lemma 3.20 below]{AMY2019}). In the particular situation of the bounded derived category of a finite-dimensional algebra, these two kinds of objects correspond bijectively to bounded co-$t$-structures, silting objects \cite[Theorem 6.1]{KY2014}. Mutation of simple-minded collections was defined in \cite{KY2014}, it is compatible with the mutation of $t$-structures, co-$t$-structures and silting objects. In the context of Bridgeland stability conditions, mutations, or tilts, of simple-minded collections correspond to wall crossing \cite{KQ2015}.

Recollements of triangulated categories,  introduced by Beilinson, Bernstein and
Deligne in their fundamental work on perverse sheaves \cite{BBD1982}, play an important role in studying the representation theory of algebras \cite{HKLY2017, CX2017}. Whenever a recollement is given, one may, and often should, ask how particular objects or properties in the outer terms are related to similar kinds of objects or properties in the middle term. This process is usually named gluing or reduction (along a recollement) in literature and has been {investigated in such cases }as reduction of the finiteness of global or finitistic dimensions \cite{CX2017, Happ1993}, or Grothendieck groups or $K$-theory of finite-dimensional algebras \cite{HKLY2017, CX2013} and gluing of tilting objects \cite{HKL2011}, $t$-structures \cite{BBD1982}, co-$t$-structures \cite{Bond2010} and silting objects \cite{AI2012, LVY2014}. Gluing of simple-minded collections is obtained implicitly in \cite[Proposition 5.5]{CPP2022}, which is an explicit specialization of \cite[Theorem 3.1]{J2023}. In this paper, we follow the ideas of reduction in \cite{CPP2022} and gluing silting objects in \cite{LVY2014} to present an explicit construction of simple-minded collections along a recollement of triangulated categories. Our gluing (Theorem~\ref{thm:intro 1}) is related to Coelho {\color{black}Sim\~oes}-Pauksztello-Ploog's reduction \cite[Proposition 5.5]{CPP2022} for simple-minded collections in a way analogous to the way that Liu-Vit\'{o}ria-Yang's gluing \cite[Theorem 3.1]{LVY2014} is related to Aihara-Iyama's reduction \cite[Theorem 2.37]{AI2012} for silting objects. The provided gluing technique for simple-minded collections is also shown to be compatible with that of $t$-structures in the setting of Hom-finite Krull-Schmidt triangulated categories, with those of co-$t$-structures and silting objects in the setting of bounded derived categories of finite dimensional algebras. Furthermore, we show it preserves partial order and commutes with the operation of mutation of simple-minded collections.

Now, let's state the main results of this paper.
Let $\T$ be a Hom-finite Krull-Schmidt triangulated category and a recollement (see Subsection~\ref{subsec:recollement}) of triangulated categories $\X$ and $\Y$. Denote the two adjoint triples by $(i^*, i_*, i^!)$ and $(j_!, j^!, j_*)$. Suppose that $\X$ and $\Y$ admit simple-minded collections (see Subsection~\ref{subsec:smc}) $S_\X=\{X_1, \cdots, X_m\}$ and $S_\Y=\{Y_1, \cdots, Y_n\}$, respectively. There is an example {\color{black}showing} that neither of the generating sets $\{i_*(X_1),\cdots, i_*(X_m),  j_*(Y_1), \cdots,  j_*(Y_n)\}$ and $\{i_*(X_1),\cdots, i_*(X_m),  j_!(Y_1), \cdots,  j_!(Y_n)\}$ are simple-minded collections (Example~\ref{exam:non smc}). Our first main result is to present so-called gluing techniques to refine these sets to be simple-minded collections of $\T$,  the two processes are dual and the refined simple-minded collections are shown to be isomorphic by comparing their associated $t$-structures.
The following theorem is due to \cite{CPP2022, J2023}, here we restate it in the language of recollements.

\begin{thm}{\textnormal{({\color{black}\cite[Proposition 5.5]{CPP2022}, \cite[Theorem 2.37]{J2023}}, Theorem~\ref{thm:main}, Proposition~\ref{prop:comp $t$-structure})}}\label{thm:intro 1} With the notation as above. The set
$$S_\T=\{i_{*}(X_{1}),i_{*}(X_{2}),\cdots, i_{*}(X_{m}),W_{1},W_{2},\cdots,W_{n}\}$$
is a simple-minded collection of $\T$, where $W_j$ lies in the triangle (see Section~\ref{sec:gluing} for the detailed construction)
        $$i_*\tau^{\leq 0}i^!j_!(Y_j)\longrightarrow j_!(Y_j)\longrightarrow W_j\longrightarrow i_*\tau^{\leq 0}i^!j_!(Y_j)[1],$$
where $\tau^{\leq 0}$ is the truncation at 0 of the associated $t$-structure $(\Filt S_\X[\geq 0], \Filt S_\X[\leq 0])$ in $\X$. Moreover, the gluing technique presented is compatible with that of $t$-structures.
\end{thm}

The above gluing technique for simple-minded collections enjoys good properties. Firstly, it preserves partial order (see Section~\ref{sec:pre}) of simple-minded collections, that is if $S_\X\geq S'_\X$ or $S_\Y\geq S'_\Y$, then the glued simple-minded collections $\langle S_\X, S_\Y\rangle \geq \langle S'_\X, S_\Y\rangle, \langle S_\X, S_\Y'\rangle\geq \langle S'_\X, S'_\Y\rangle$ (Theorem~\ref{thm:preserve order}), where the notation $\langle S, S'\rangle$ denotes the glued simple-minded collection of $S$ and $S'$. Furthermore, it commutes with the operation of mutation. Denote the left (resp., right) mutation (see Section~\ref{sec:pre}) in $i$ by $\mu^+_i$ (resp., $\mu^-_i$).

\begin{thm}[Theorem~\ref{thm:mutation smc}]\label{thm:intro 2}
With the notation in Theorem~\ref{thm:intro 1} and as above.
  \begin{enumerate}
  \item Suppose that $X_i$ is rigid. Then
   $$\langle \mu^+_i(S_\X), S_\Y\rangle=\mu_i^+\langle S_\X,  S_\Y\rangle\quad \text{and}\quad \langle \mu^-_i(S_\X), S_\Y\rangle=\mu_i^-\langle S_\X,  S_\Y\rangle.$$
  \item Suppose that $Y_j$ is rigid.
   \begin{enumerate}
     \item If $\Hom(i_*(X_t), W_j[1])=0, 1\leq t\leq m$,  then $$\langle S_\X, \mu^+_j(S_\Y)\rangle=\mu^+_{m+j}(\langle S_\X, S_\Y\rangle).$$
     \item If $\Hom(W_j, i_*(X_t)[1])=0, 1\leq t\leq m$,  then $$\langle S_\X, \mu^-_j(S_\Y)\rangle=\mu^-_{m+j}(\langle S_\X, S_\Y\rangle).$$
      \end{enumerate}
\end{enumerate}
\end{thm}
The conditions in $(a), (b)$ are necessary for ensuring the simple-minded collections on both sides admit the same first $m$-{terms.}

Finally, we show our gluing technique for simple-minded collections is compatible with those for silting objects, $t$-structures, and co-$t$-structures {under mild conditions.} Indeed, the compatibility between gluing silting objects and gluing co-$t$-structures is shown in \cite{LVY2014} and the compatibility between gluing simple-minded collections and gluing $t$-structures is shown in {Theorem~\ref{thm:intro 1}}. So to reach our goal, it suffices to show the compatibility between gluing $t$-structures and co-$t$-structures, this has been proved in \cite[Proposition 6.14]{SZ2022} with a mild condition.

The contents of this paper are organized as follows.
In Section~\ref{sec:pre}, we {\color{black}fix} notation and recall some basic concepts involved. In Section~\ref{sec:gluing}, we present the methods for gluing simple-minded collections and prove Theorem~\ref{thm:intro 1}. In Section~\ref{sec:order}, we compare the gluing technique with order and the basic operation of mutation of simple-minded collections and prove Theorem~\ref{thm:intro 2}. ~\\

{\em Acknowledgements.} The authors would like to thank Prof. Bin Zhu for stimulating discussions and pointing out typographical errors.  They are grateful to Prof. Dong Yang for answering our questions and Prof. 
Alexandra Zvonareva for introducing us to the paper \cite{SZ2022}. They gratefully thank the referee for his/her very helpful comments.

\section{Preliminaries}\label{sec:pre}
\subsection{{Notation}}
In this paper,  {$K$ will denote} an algebraically closed field. All triangulated categories are assumed to be Hom-finite Krull-Schmidt over the {  field} $K$. Denote the shift functors of all the triangulated categories by $[1]$. For convenience, the functor $\Hom(-, -)$ in some diagrams is denoted by $(-, -)$ simply.
The composition of $f\in\Hom(X, Y)$ and $g\in\Hom(Y, Z)$ is defined by $gf\in\Hom(X, Z)$. Let $\C$ be a category. For $X\in\C$, $\add (C)$ denotes the smallest full subcategory of $\C$ containing $X$ and closed under direct summands of finite direct sums. Let $f: M\to N$ be a morphism. $f$ is a {\em left $\add(X)$-approximation} if $N\in \add(X)$ and the induced map $f_*:\Hom(N, X)\to \Hom(M, X)$ is surjective. $f$ is {\em left minimal} if for any $g: N\to N$ satisfies $gf=f$, then $g$ {  is an isomorphism.} $f$ is called a {\em minimal left approximation} if it is both a left approximation and left minimal. Right approximation, right minimal, and minimal right approximation are defined dually.

Let $A$ be a finite dimensional $K$-algebra, denote $\Mod A$ (resp., $\smodA, \projA$) the category of right $A$-modules (resp., finite-dimensional right $A$-modules, finite dimensional projective right $A$-modules), denote $\D(\Mod A)$ (resp., $\D^b(\smodA)$, $\K^b(\projA)$) the derived category (resp., bounded derived category, homotopy category of bounded complexes of $\projA$).

\subsection{Recollements}\label{subsec:recollement}
Let $\mathscr{T}$ be a triangulated category and $\X, \Y$ be triangulated subcategories of $\mathscr{T}$.
We say that $\mathscr{T}$ is a {\em recollement} \cite{BBD1982} of $\X$
and $\Y$ if there are six triangle functors as in the diagram
$$\xymatrix{\X\ar^-{i_*=i_!}[rr]&&\mathscr{T}\ar^-{j^!=j^*}[rr]
\ar^-{i^!}@/^1.2pc/[ll]\ar_-{i^*}@/_1.6pc/[ll]
&&\Y\ar^-{j_*}@/^1.2pc/[ll]\ar_-{j_!}@/_1.6pc/[ll]}$$
such
that
\begin{enumerate}
  \item $(i^*,i_*),(i_!,i^!),(j_!,j^!)$ and $(j^*,j_*)$ are adjoint pairs;
  \item $i_*,j_*$ and $j_!$ are fully faithful functors;
  \item $i^!j_*=0$ (and thus also $j^! i_!=0$ and $i^*j_!=0$); and
  \item for each object $T\in\mathscr{T}$, there are two triangles in $\T$
$$i_!i^!(T)\to T\to j_*j^*(T)\to i_!i^!(T)[1],$$
$$j_!j^!(T)\to T\to i_*i^*(T)\to j_!j^!(T)[1].$$
\end{enumerate}

\subsection{$t$-structures}
A pair $(\T^{\leq 0}, \T^{\geq 0})$ of subcategories of $\T$ is called a
{\em $t$-structure} \cite{BBD1982} if it satisfies
\begin{enumerate}
  \item $\T^{\leq 0}[1]\subset \T^{\leq 0}$ and $\T^{\geq 0}[-1]\subset \T^{\geq 0}$;
  \item $\Hom(\T^{\leq 0}, \T^{\geq 0}[-1])=0$;
  \item For any $T\in\T$, there exists a triangle
  $$U\longrightarrow T\longrightarrow V\longrightarrow U[1],$$
 where $U\in\T^{\leq 0}$ and $V\in\T^{\geq 0}[-1]$.
\end{enumerate}
For a $t$-structure $(\T^{\leq 0}, \T^{\geq 0})$, the subcategory {  $\T^{\leq 0}\bigcap\T^{\geq 0}$ }is called the {\em heart} of the $t$-structure. It is an abelian category (\cite{BBD1982}).
Set $\T^{\leq -i}=\T^{\leq 0}[i]$ and $\T^{\geq -i}=\T^{\geq 0}[i]$.
A $t$-structure is {\em bounded} if $\bigcup_{i\in\mathbb{Z}}\T^{\leq i}=\T=\bigcup_{i\in\mathbb{Z}}\T^{\geq i}$. { A bounded $t$-structure is called {\em algebraic} ({\color{black}\cite[Section 2.3]{W2010}},\cite[after Lemma 3.20]{AMY2019}) if it has a length heart with finitely many isomorphism classes of simple objects.}

Gluing $t$-structures has been studied in \cite[Theorem 1.4.10]{BBD1982}. Suppose that $\T$ is a
recollement of triangulated categories $\X$ and $\Y$. Given $t$-structures $(\X^{\leq 0}, \X^{\geq 0})$ and $(\Y^{\leq 0}, \Y^{\geq 0})$ of $\X$ and $\Y$, then we get a glued $t$-structure $(\T^{\leq 0}, \T^{\geq 0})$ on $\T$, where
 \begin{align*}
    \T^{\leq 0}=& \{T\in\T\mid i^{*}(T)\in \X^{\leq 0}, j^{!}(T)\in \Y^{\leq 0}\}, \\
    \T^{\geq 0}=& \{T\in\T\mid i^{!}(T)\in \X^{\geq 0}, j^{!}(T)\in \Y^{\geq 0}\}.
 \end{align*}

 The following easy lemma explains the relation among these $t$-structures.

\begin{lem}\label{lem:glue t}
 {$i_*(\X^{\leq 0})\subseteq \T^{\leq 0}, j_!(\Y^{\leq 0})\subseteq \T^{\leq 0}$ and $i_*(\X^{\geq 0})\subset\T^{\geq 0}, j_*(\Y^{\geq 0})\subseteq \T^{\geq 0}$;}
\end{lem}

\subsection{Simple-minded collections}\label{subsec:smc}
Let $\T$ be a triangulated category. A  collection $S=\{S_{1},S_{2},\cdots,S_{n}\}$  of objects in $\T$ is called a {\em simple-minded collection} (cohomologically Schurian set of generators in\cite{AI2009}, \cite[Hypothesis]{RR2017}) if
\begin{enumerate}
  \item $\mathsf{dim}_K\Hom(S_{i},S_{j})={\delta_{ij}}$, where $\delta$ is the Kronecker function;
  \item {$\S$} generates $\T$ as a triangulated category, i.e., $\T$ is the smallest thick triangulated subcategory containing {$\S$};
  \item $\Hom(S_{i}, S_{j}[n])=0$ for $n<0$.
\end{enumerate}
Two simple-minded collections are said {\em isomorphic}\footnote{In Koenig and Yang's paper \cite{KY2014}, they used "equivalent", but when we asked Yang about this notion, he said they should have used "isomorphic".} \cite{KY2014} if they are the same up to an isomorphism. It is obvious that a set $S$ is a simple-minded collection in $\T$ if and only if so is in the opposite category $\T^{\op}$.

In a Hom-finite Krull-Schmidt triangulated category, there is a one-to-one correspondence between isomorphism classes of
simple-minded collections and algebraic $t$-structures (See \cite[Section 5.3, 5.5]{KY2014}). Let $S$ be a simple-minded collection of $\T$, then the corresponding algebraic $t$-structure is
$(\Filt S[\geq 0], \Filt S[\leq 0])$, where
 \begin{align*}
   \Filt S[\geq 0]= &~ \text{extension closure of}~S[m],~m\geq 0, \\
   \Filt S[\leq 0]= &~ \text{extension closure of}~S[m],~m\leq 0.
 \end{align*}
 This construction is given firstly by AI-Nofayee \cite[Corollary 3 and Proposition 4]{AI2009} in the setting of derived categories.
 Let $(\T^{\leq 0}, \T^{\geq 0})$ be an algebraic $t$-structure of $\T$, then a set of non-isomorphic simple objects of its heart is a simple-minded collection of $\T$. Mutations of algebraic $t$-structures are defined in \cite[Section 7.3]{KY2014}, and it is compatible with those of simple-minded collections.

\subsection{Co-$t$-structures}
 A pair $(\T_{\geq 0}, \T_{\leq 0})$ of subcategories of $\T$ is called a
{\em co-$t$-structure} \cite{Bond2010, Pauk2008} if it satisfies
\begin{enumerate}
\item both $\T_{\geq 0}$ and $\T_{\leq 0}$ are additive and closed under direct summands,
  \item $\T_{\geq 0}[-1]\subset \T_{\geq 0}$ and $\T_{\leq 0}[1]\subset \T_{\leq 0}$,
  \item $\Hom(M, N[1])=0$ for $M\in\T_{\geq 0}$ and $N\in \T_{\leq 0}$
  \item  for $T\in\T$, there exists a triangle
  $$U\longrightarrow T\longrightarrow V\longrightarrow U[1]$$
 where $U\in\T_{\geq 0}$ and $V\in\T_{\leq 0}[1]$.
\end{enumerate}
The co-$t$-structure is {\em bounded }if $\bigcup_{i\in\mathbb{Z}}\T_{\geq 0}[i]=\T=\bigcup_{i\in\mathbb{Z}}\T_{\leq 0}[i]$.

Gluing co-$t$-structures has been studied in \cite[Section 8.2]{Bond2010}. Suppose that $\T$ is a
recollement of triangulated categories $\X$ and $\Y$.
Given co-$t$-structure $(\X_{\geq 0}, \X_{\leq 0})$ on $\X$ and co-$t$-structure $(\Y_{\geq 0}, \Y_{\leq 0})$ on $\Y$, then there is a glued co-$t$-structure $(\T_{\geq 0}, \T_{\leq 0})$, where
\begin{align*}
 \T_{\geq 0}= & \{V\in \T\mid i^{*}(V)\in \X_{\geq 0}, j^{!}(V)\in \Y_{\geq 0}\} \\
 \T_{\leq 0}= & \{U\in \T\mid i^{!}(U)\in \X_{\leq 0}, j^{!}(U)\in \Y_{\leq 0}\}
\end{align*}

The following easy lemma explains the relation among these co-$t$-structures.

\begin{lem}\label{lem:glue co-t}
 {$i_*(\X_{\geq 0})\subseteq \T_{\geq 0},  j_!(\Y_{\geq 0})\subseteq \T_{\geq 0}$ and $i_*(\X_{\leq 0})\subseteq \T_{\leq 0}, j_*(\Y_{\leq 0})\subseteq \T_{\leq 0}$}.
\end{lem}

\section{Gluing simple-minded collections}\label{sec:gluing}
{ The gluing method presented in this section is obtained implicitly in \cite{CPP2022}. Following the ideas of reduction in \cite{CPP2022} and gluing silting objects in \cite{LVY2014}, we will give an explicit construction of simple-minded collections in a recollement.} Let $\T$ be a triangulated category. Assume $\T$ is a recollement of triangulated categories $\X$ and $\Y$
$$\xymatrix{\X\ar^-{i_*=i_!}[rr]&&\mathscr{T}\ar^-{j^!=j^*}[rr]
\ar^-{i^!}@/^1.2pc/[ll]\ar_-{i^*}@/_1.6pc/[ll]
&&\Y\ar^-{j_*}@/^1.2pc/[ll]\ar_-{j_!}@/_1.6pc/[ll]}.$$

\begin{lem}\label{lem:set from reco}
  Let $S_{\X}=\{X_{1},X_{2},\cdots,X_{m}\}, S_{\Y}=\{Y_{1},Y_{2},\cdots,Y_{n}\}$ be simple-minded collections in $\X$ and $\Y$ respectively. Then
  \begin{enumerate}
   \item $\Hom(i_*(X_i), j_*(Y_j))=0=\Hom(j_!(Y_j), i_*(X_i)),~1\leq i\leq m, 1\leq j\leq n$;
    \item both $\{i_*(X_1),\cdots, i_*(X_m),  j_*(Y_1), \cdots,  j_*(Y_n)\}$ and $\{i_*(X_1),\cdots, i_*(X_m),$\\
        $ j_!(Y_1), \cdots,  j_!(Y_n)\}$ generate $\T$.
  \end{enumerate}
\end{lem}
\begin{proof}
By a routine check.
\end{proof}

The collections in Lemma~\ref{lem:set from reco}(2) satisfy the condition (2) in the definition of simple-minded collections. But in general, they do not satisfy the conditions (1) or (3). Let's give an example.

\begin{exam}\label{exam:non smc}
Let $A$ be the algebra given by the quiver
$$
\xymatrix{1~\bullet\ar@<-0.8ex>[r]_{\beta}
&\bullet~2\ar[l]_{\alpha}
}$$
with the relation $\beta\alpha=0$. Note that $A$ has a finite global dimension. Set $e=e_1$, $A\to A/AeA$ is a homological epimorphism (since $_AAeA$ is projective). Then we have the following  recollement (for instance, see \cite[Subsection 2.1]{HKLY2017}, where $eAe\simeq k$ and $A/AeA\simeq k$
$$\xymatrix{\D^{b}(\smodAeAe)\ar^{-\otimes_{A/AeA}A/AeA}[r]&\D^{b}(\smodA)
\ar^-{-\otimes^{L}_{A}Ae}[r]
\ar^-{\RHom_{A}(A/AeA,-)}@/^1.2pc/[l]\ar_-{-\otimes^{L}_{A}A/AeA}@/_1.6pc/[l]
&\D^{b}(\smod eAe)\ar^-{\RHom_{eAe}(Ae,-)}@/^1.2pc/[l]\ar_-{-\otimes^{L}_{eAe}eA}@/_1.6pc/[l]}$$

Let $S_{1}$ and $S_{2}$ be the simple modules of $\smodA$. Then $S_1$ and $S_2$ are also simple $eAe$-module and $A/AeA$-module, respectively. Furthermore, there are isomorphisms
\begin{center}
  \begin{tabular}{l}
   $S_{2}\otimes_{A/AeA}A/AeA\simeq S_{2}$,\\
  $S_{1}\otimes^{L}_{eAe}eA\simeq eA$,\\
  $\RHom_{eAe}(Ae, S_{1})\simeq I_{1}$,
  \end{tabular}
\end{center}
where $I_1$ is the indecomposable injective module corresponding to $e$. Since  $$\Hom_{A}(S_{2}, eA)\neq 0~\text{and}~\Hom_{A}(I_{1},S_{2})\neq 0.$$
Then the sets
$\{S_{2}, I_1\}~\text{and}~\{S_{2}, eA\}$
are not simple-minded collections.
\end{exam}

As indicated in the above example, to obtain a simple-minded collection of $\T$ from $S_\X$ and $S_\Y$, we need to refine the set
$\{i_*(X_1),\cdots, i_*(X_m),  j_*(Y_1), \cdots,  j_*(Y_n)\}$ or the set $\{i_*(X_1),\cdots, i_*(X_m),  j_!(Y_1), \cdots,  j_!(Y_n)\}$.
We refine the first set {first}. { Note that, the set $i_*(S_\X)=\{i_*(X_1),\cdots, i_*(X_m)\}$ determines a $t$-structure $(\mathsf{Filt}~i_*(S_\X)[\geq 0],(\mathsf{Filt}~i_*(S_\X)[\geq 0])^{\perp}[1])$ in $\T$ because $\mathsf{Filt}~S_\X[\geq 0]$ is contravariantly finite in $\mathcal{X}$ as the left-hand part of the $t$-structure $(\mathsf{Filt}~S_\X[\geq 0],\mathsf{Filt}~S_\X[\leq 0])$ in $\mathcal{X}$, and $i_*(\mathcal{X})$ is right admissible (or coreflective) in $\mathcal{T}$. Hence, following the idea of the proof in \cite[Proposition 5.5]{CPP2022}, to obtain a simple-minded collection of $\T$ from a given one $S_\Y$ of $\Y$, we need to truncate each object in $S_\Y$ with respect to the $t$-structure above. To be explicit,} we divide the process into three steps. 

{\bf Step 1:} Take the canonical decomposition for each $j_!(Y_i)$,
$$
i_!i^!(j_!(Y_i))\s{f}\longrightarrow j_!(Y_i)\longrightarrow j_*j^*(j_!(Y_i)) \simeq j_*(Y_i)\longrightarrow i_!i^!(j_!(Y_i))[1].
$$

{\bf Step 2:} Take the decomposition of $i^!j_!(Y_i)$ with respect to the $t$-structure $(\Filt S_\X[\geq 0], \Filt S_\X[\leq 0])$,
$$U_i\s{g}\longrightarrow i^!j_!(Y_i) \longrightarrow V_i \longrightarrow U_i[1].$$
where $U_i\in \Filt S_\X[\geq 0]$ and $V_i\in \Filt S_\X[\leq -1]$.

{\bf Step 3:} Use the octahedral axiom of triangulated categories, there is the following commutative diagram in $\T$, where the triangle in the first row  is the image of the triangle in Step 2 under $i_*$, and the triangle in the second column is the triangle in Step 1
$$\xymatrix{
i_{*}(U_{i})\ar[r]^{i_*(g)}\ar@{=}[d]
&i_{*}i^{!}j_{!}(Y_{i})\ar[r]\ar[d]^{f}
&i_{*}(V_{i})\ar[r]\ar[d]
&i_{*}(U_{i})[1]\ar@{=}[d]
\\
i_{*}(U_{i})\ar[r]^{fi_*(g)}
&j_{!}(Y_{i})\ar[r]\ar[d]
&W_{i}\ar[r]\ar[d]
&i_{*}(U_{i})[1]
\\
&j_*(Y_i)\ar@{=}[r]
&j_*(Y_i).
&
}$$
Hence, we get $W_i$ from $Y_i$ and obtain a new collection $\{i_*(X_1), \cdots, i_*(X_m), W_1, \cdots, W_n\}$ which is shown to be a simple-minded collection of $\T$.

For convenience, we mark the following triangles
\begin{align*}
  (\#) & \quad i_*(U_i)\s{fi_*(g)}\longrightarrow j_!(Y_i) \longrightarrow W_i \longrightarrow i_*(U_i)[1] \\
  (\#\#) & \quad i_*(V_i)\longrightarrow W_i\longrightarrow j_*(Y_i)\longrightarrow i_*(V_i)[1]
\end{align*}

\begin{lem}\label{lem:image of W}
The following hold.
\begin{enumerate}
  \item $j^!(W_i)\simeq Y_i$;
  \item $i^*(W_i)\simeq U_i[1]$;
  \item $i^!(W_i)\simeq V_i$.
\end{enumerate}
\end{lem}
\begin{proof}
  Consider the images of the triangles $(\#)$ and $(\#\#)$ under the functors $j^!, i^*$ and $i^!$,  {and using the fact that the unit (resp. counit) of an adjunction is an isomorphism if and only if the left (resp. right) adjoint is fully faithful.}
\end{proof}

\begin{lem}\label{lem:prep1}
$\Hom_{\T}(W_{i},W_{j}[t])\simeq \Hom_{\Y}(Y_{i},Y_{j}[t]),~1\leq i,j\leq n, ~t\leq 0$.
\end{lem}
\begin{proof}
By applying $\Hom(-,W_{j}[t])$ { to} the triangle $(\#)$,
there is an exact sequence (denote $\Hom(-, -)$ simply by $(-, -)$)
$$\cdots\to (i_{*}(U_{i})[1], W_{j}[t])\to (W_{i},W_{j}[t])\to(j_{!}(Y_{i}),W_{j}[t])\to (i_{*}(U_{i}), W_{j}[t])\to\cdots.$$

{\bf Claim 1:} $\Hom(W_{i},W_{j}[t])\simeq \Hom(j_{!}(Y_{i}),W_{j}[t])$ for $t\leq 0$.

By applying $\Hom(i_{*}(U_{i}),-), \Hom(i_{*}(U_{i})[1], -[t])$ { to} the triangle $(\#\#)$,
then there are exact sequences
$$(i_{*}(U_{i}), i_{*}(V_{j})[t])\to (i_{*}(U_{i}), W_{j}[t])\to (i_{*}(U_{i}), j_{*}(Y_{j})[t])$$
and
$$(i_{*}(U_{i})[1], i_{*}(V_{j})[t])\to (i_{*}(U_{i})[1], W_{j}[t])\to (i_{*}(U_{i})[1], j_{*}(Y_{j})[t]).$$
Since $\Hom(i_{*}(U_{i}), i_{*}(V_{j})[t])= 0=\Hom(i_{*}(U_{i})[1], i_{*}(V_{j})[t])$ for $t\leq 0$ by the properties of $t$-structure, and $\Hom(i_{*}(U_{i})[k], j_{*}(Y_{j})[t])\simeq \Hom(j^*i_{*}(U_{i})[k],Y_{j}[t])=0$ for $k=0, 1$,
then we have
$$\Hom_{\T}(i_{*}(U_{i})[1],W_{j}[t])\simeq 0\simeq\Hom_{\T}(i_{*}(U_{i}),W_{j}[t]),~t\leq 0.$$
Therefore, the claim holds.

{\bf Claim 2:} $\Hom_{\T}(W_{i},W_{j}[t])\simeq \Hom_{\Y}(Y_{i},Y_{j}[t])$ for $t\leq 0$.

By applying $\Hom_{\T}(j_{!}(Y_{i}),-[t])$ { to} the triangle $(\#)$,
we obtain the following exact sequence
$$(j_{!}(Y_{i}), i_{*}(U_{j})[t])\to (j_{!}(Y_{i}), j_{!}(Y_{j})[t])\to (j_{!}(Y_{i}), W_{j}[t])\to (j_{!}(Y_{i}), i_{*}(U_{j})[t+1]),$$
Since $(j_{!}(Y_{i}), i_{*}(U_{j})[k])\simeq (Y_{i}, j^*i_{*}(U_{j})[k])=0, k=t, t+1$, then
$$\Hom_{\T}(j_{!}(Y_{i}), W_{j}[t])\simeq \Hom_{\T}(j_{!}(Y_{i}), {j_{!}(Y_{j})[t]})\simeq \Hom_{\Y}(Y_{i},Y_{j}[t]),~\forall t\in\mathbb{Z}.$$
It follows from claim 1 that claim 2 holds. We finish the proof.
\end{proof}

\begin{lem}\label{lem:prep2}
 For $1\leq i\leq m, 1\leq j\leq n, t\leq 0$, then
 $$\Hom_{\T}(i_{*}(X_{i}), W_{j}[t])= 0=\Hom_{\T}(W_{j}, i_{*}(X_{i})[t]).$$
\end{lem}
\begin{proof}
 By applying $\Hom_{\T}(i_{*}(X_{i}),-)$ to the $t$-times shift of the triangle $(\#\#)$, we have an exact sequence
$$(i_{*}(X_{i}), i_*(V_i)[t])\longrightarrow (i_{*}(X_{i}), W_i[t])\longrightarrow (i_{*}(X_{i}), j_*(Y_i)[t]).$$
Note that {$\Hom(i_{*}(X_{i}), j_*(Y_i)[t])=\Hom(j^*i_*(X_i), Y_i[t])=0$} and for $t\leq 0, V_i[t]\in  \Filt S_\X[\leq -1]$, so $\Hom(i_{*}(X_{i}), i_*(V_i)[t])\simeq \Hom(X_{i}, V_i[t])=0$. Hence we have
 $$\Hom(i_{*}(X_{i}), W_{j}[t])=0.$$

By applying $\Hom_{\T}(-, i_{*}(X_{i})[t])$ { to} the triangle $(\#)$,
we have the exact sequence
$$(i_*(U_i)[1],  i_{*}(X_{i})[t])\longrightarrow (W_i,  i_{*}(X_{i})[t]) \longrightarrow (j_!(Y_i), i_{*}(X_{i})[t]).$$
Note that $\Hom(j_!(Y_i), i_{*}(X_{i})[t])\simeq \Hom(Y_i, j^*i_*(X_i)[t])=0$ and for $t\leq 0,~X_i[t]\in \Filt S_\X[\leq 0]$, so $\Hom(i_*(U_i)[1],  i_{*}(X_{i})[t])\simeq \Hom(U_i[1],  X_{i}[t])=0$ { since $U_i[1]\in \Filt S_\X[\geq 1]$}. Hence, we have
$$\Hom_{\T}(W_{j}, i_{*}(X_{i})[t])=0, t\leq 0.$$
This finishes the proof.
\end{proof}

Set $S_{\T}:=\{i_{*}(X_{1}),i_{*}(X_{2}),\cdots,i_{*}(X_{m}),W_{1},W_{2},\cdots,W_{n}\}$. Based on the above lemmas, we can prove the main theorem below.

\begin{thm}{\color{black}\textnormal{(\cite[Proposition 5.5]{CPP2022}, \cite[Theorem 3.1]{J2023})}}\label{thm:main}
$S_{\T}$ is a simple-minded collection in $\T$.
\end{thm}
\begin{proof}
It follows from Lemma~\ref{lem:prep1} and Lemma~\ref{lem:prep2} that $S_{\T}$ satisfies the conditions (1) and (3) in the definition of simple-minded collections.
Since there exists the triangle $(\#)$ and $i_*(U_j)$ can be generated by $\{{i_{*}(X_1)}, \cdots, i_*(X_m)\}$, then each $j_{!}(Y_{j})$ can be generated by $S_{\T}$. By Lemma~\ref{lem:set from reco}, we know $S_{\T}$ generates $\T$. Therefore $S_\T$ is a simple-minded collection of $\T$.
\end{proof}

The second set can be refined to a simple-minded collection dually.
Indeed, we can refine the set $\{ i_*(X_1), \cdots, i_*(X_m), j_*(Y_1), \cdots,  j_*(Y_n) \}$ through the dual three steps.

{\bf Step 1':} Consider the canonical triangle of $j_*(Y_i)$,
$$j_!j^*j_*(Y_i)\simeq j_!(Y_i)\longrightarrow j_*(Y_i)\longrightarrow i_*i^*j_*(Y_i)\longrightarrow j_!j^*j_*(Y_i)[1]$$

{\bf Step 2':} Decompose $i^*j_*(Y_i)$ with respect to the corresponding algebraic $t$-structure $(\Filt S_\X[\geq 1], \Filt S_\X[\leq 1])$
$$M_i\longrightarrow i^*j_*(Y_i)\longrightarrow N_i\longrightarrow M_i[1]$$
where $M_i\in \Filt S_\X[\geq 1]$ and $N_i\in \Filt S_\X[\leq 0]$.

{\bf Step 3':} Use the octahedral axiom of triangulated categories, where the triangle in the third row is the image of the triangle in step 2' under $i_*$ and the triangle in the second column is from step 1',
$$\xymatrix{
&j_!(Y_i)\ar@{=}[r]\ar[d]
&j_!(Y_i)\ar[d]
&
\\
i_*(N_i)[-1]\ar@{=}[d]\ar[r]
&P_i\ar[d]\ar[r]
&j_*(Y_i)\ar[d]\ar[r]
&i_*(N_i)\ar@{=}[d]
\\
i_*(N_i)[-1]\ar[r]
&i_*(M_i)\ar[r]
& i_*i^*j_*(Y_i)\ar[r]
&i_*(N_i)
}$$

{For convenience, we mark the following triangles
\begin{align*}
  (\diamondsuit) & \quad i_*(N_i)[-1]\longrightarrow P_i \longrightarrow j_*(Y_i) \longrightarrow i_*(U_i) \\
  (\diamondsuit\diamondsuit) & \quad j_!(Y_i)\longrightarrow P_i\longrightarrow i_*(M_i)\longrightarrow j_!(Y_i)[1]
\end{align*}}

Set $S^\T:=\{i_*(X_1), \cdots, i_*(X_m), P_1, \cdots, P_n\}$. A dual discussion can prove that $S^\T$ is a simple-minded collection in $\T$. Here, we present an {\color{black}alternative} way to explain why this is true.
Note that the opposite of the recollement $(\X, \T, \Y)$ is of the form
$$\xymatrix{\X^{\op}\ar^-{i_*^{\op}=i_!^{\op}}[rr]
&&\mathscr{T}^{\op}\ar^-{{j^!}^{\op}={j^*}^{\op}}[rr]
\ar_-{{i^!}^{\op}}@/_1.6pc/[ll]\ar^-{{i^*}^{\op}}@/^1.2pc/[ll]
&&\Y^{\op}\ar_-{j_*^{\op}}@/_1.6pc/[ll]\ar^-{j_!^{\op}}@/^1.2pc/[ll]}.$$ 
$S_\X, S_\Y$ are also simple-minded collections in $\X^{\op}$ and $\Y^{\op}$ respectively. The glued simple-minded collection from $S_\X$ and $S_\Y$ along the opposite recollement by using step 1---step 3 coincides with the set $S^\T$, hence $S^\T$ is a simple-minded collection in $\T^{\op}$, so {\color{black}it is also a simple-minded collection in $\T$}.

By now, we have constructed two simple-minded collections $S_\T$ and $S^\T$. The relation between $S_\T$ and $S^\T$ is not obvious. We will show they are the same up to an isomorphism by comparing their corresponding $t$-structures. By $(\T^{\leq 0}, \T^{\geq 0})$ we denote the glued $t$-structure of $(\Filt S_\X[\geq 0], \Filt S_\X[\leq 0])$ and $(\Filt S_\Y[\geq 0], \Filt S_\Y[\leq 0])$ along the recollement.

\begin{lem}\label{lem:prep4.1}
The following hold.
\begin{enumerate}
  \item $j_!(\Filt S_\Y[\geq 0])\subset \Filt S_\T[\geq 0]$;
  \item $j_!(\Filt S_\Y[\geq 0])\subset \Filt S^\T[\geq 0]$;
  \item $S_\T\subset \T^{\leq 0}$;
  \item $S^\T\subset \T^{\leq 0}$.
\end{enumerate}
\end{lem}
\begin{proof}
(1)
Consider the triangle $(\#)$, i.e.
  $$i_*(U_i)\s{fi_*(g)}\longrightarrow j_!(Y_i) \longrightarrow W_i \longrightarrow i_*(U_i)[1].$$
  Since $i_*(U_i)\in i_*(\Filt S_\X[\geq 0])\subset \Filt S_\T[\geq 0], W_i\in \Filt S_\T[\geq 0]$ and $\Filt S_\T[\geq 0]$ is closed under extensions, then  $j_!(Y_i)\in\Filt S_\T[\geq 0]$. Hence the statement holds.

(2)
In the following triangle (i.e the shift  of the triangle $(\diamondsuit\diamondsuit)$)
$$i_*(M_i)[-1]\longrightarrow j_!(Y_i)\longrightarrow P_i\longrightarrow i_*(M_i),$$
$i_*(M_i)[-1]\in i_*(\Filt S_\X[\geq 0])\subset \Filt S^\T[\geq 0]$ and $P_i\in\Filt S^\T[\geq 0]$, then $j_!(Y_i)\in \Filt S^\T[\geq 0]$. Hence we have $j_!(\Filt S_\Y[\geq 0])\subset \Filt S^\T[\geq 0]$

(3)
Obviously, $i_*(S_\X)\subset \T^{\leq 0}$. Thanks to Lemma~\ref{lem:image of W} (1) and (2), we have $i^*(W_i)\simeq U_i[1]\in \Filt S_\X[\geq 0][1]\subseteq \Filt S_\X[\geq 0]$ and $j^!(W_i)\simeq Y_i\in \Filt S_\Y[\geq 0]$. Hence $W_i\in \T^{\leq 0}$. Therefore we have $S_\T\subset \T^{\leq 0}$.

(4)
Consider the triangle $(\diamondsuit\diamondsuit)$ above.
Since $j_!(Y_i), i_*(M_i)\in \T^{\leq 0}$ and $\T^{\leq 0}$ is closed under extensions, then $P_i\in \T^{\leq 0}$. Hence $S^\T\subset \T^{\leq 0}$, and so is $\Filt S^\T[\geq 0]$.
\end{proof}

\begin{prop}\label{prop:comp $t$-structure}
The following hold.
\begin{enumerate}
  \item $\Filt S_\T[\geq 0]=\T^{\leq 0}$;
  \item $\Filt S^\T[\geq 0]=\T^{\leq 0}$.
\end{enumerate}
\end{prop}
\begin{proof}
(1)
Since $S_\T\subset \T^{\leq 0}$  (Lemma~\ref{lem:prep4.1}), then $\Filt S_\T[\geq 0]\subseteq\T^{\leq 0}$. Conversely, let $T\in\T^{\leq 0}$, i.e. $i^*(T)\in \Filt S_\X[\geq 0]$ and $j^!(T)\in\Filt S_\Y[\geq 0]$. Consider the canonical decomposition triangle of $T$
$$j_!j^!(T)\longrightarrow T\longrightarrow i_*i^*(T)\longrightarrow j_!j^!(T)[1].$$
Since $j_!j^!(T)\in\Filt S_\T[\geq 0]$ by Lemma~\ref{lem:prep4.1} and $i_*i^*(T)\in i_*(\Filt S_\X[\geq 0])\subset \Filt S_\T[\geq 0]$, then $T\in \Filt S_\T[\geq 0]$. Hence, we have $\T^{\leq 0}\subseteq \Filt S_\T[\geq 0]$.

(2)
Since $S^\T\subset \T^{\leq 0}$  (Lemma~\ref{lem:prep4.1}), then $\Filt S^\T[\geq 0]\subseteq\T^{\leq 0}$. Conversely, let $T\in\T^{\leq 0}$. Consider the canonical triangle
$$j_!j^!(T)\longrightarrow T\longrightarrow i_*i^*(T)\longrightarrow j_!j^!(T).$$
$i_*i^*(T)\in i_*(\Filt S_\X[\geq 0])\subset \Filt S^\T[\geq 0]$. Since $j^!(T)\in\Filt S_\Y[\geq 0]$, then $j_!j^!(T)\in \Filt S^\T[\geq 0]$ by Lemma~\ref{lem:prep4.1}. Hence we have $T\in\Filt S^\T[\geq 0]$.
We finish the proof.
\end{proof}

Proposition~\ref{prop:comp $t$-structure} shows that $S_\T$ and $S^\T$ induce the same $t$-structures, by the one-to-one correspondence between algebraic $t$-structures and isomorphism classes of simple-minded collections, then $S_\T$ and $S^\T$ are isomorphic. Also, this proposition tells us the compatibility of gluing $t$-structures and gluing simple-minded collections along a recollement. The following diagram is a {framework} of our approach, where $\varphi$ is the correspondence from simple-minded collections to $t$-structures,
$$\xymatrix{
&&S_\X\ar[d]^{\varphi}\ar[ddll]_{\text{glue}}\ar[ddrr]^{\text{glue}}&&
\\
&&(\Filt S_\X[\geq 0], \Filt S_\X[\leq 0])\ar[d]^{\text{glue}}&&
\\
S_\T\ar[rr]^{\varphi}&&(\T^{\leq 0}, \T^{\geq 0})&&S^\T\ar[ll]_{\varphi}
\\
&&(\Filt S_\Y[\geq 0], \Filt S_\Y[\leq 0])\ar[u]_{\text{glue}}&&
\\
&&S_\Y\ar[u]_{\varphi}\ar[uurr]_{\text{glue}}\ar[uull]^{\text{glue}}&&
}$$

 { 
\begin{rem}
Koenig-Yang set up correspondences {\color{black}between} silting objects, co-$t$-structures, $t$-structures and simple-minded collections in \cite{KY2014}. 
Gluing co-$t$-structures is discussed in \cite{Bond2010} and gluing silting objects is studied in \cite{LVY2014}. Via Koenig-Yang's correspondences, the gluing of co-$t$-structures and silting objects is compatible by \cite{LVY2014} and the gluing of simple-minded collections and $t$-structures is compatible by Proposition~\ref{prop:comp $t$-structure}. To discuss the compatibility of these four kinds of gluing, it suffices to study the compatibility of gluing $t$-structures and co-$t$-structures. Indeed, {\color{black}Saor\'{i}n-Zvonareva \cite[Proposition 6.14]{SZ2022} showed that for finite dimensional algebras $A,B,C$, if there are the following two recollements
$$\xymatrix{\K^{b}(\projA)\ar^-{i_*}[r]&\K^{b}(\projB)\ar^-{j^!}[r]
\ar^-{i^!}@/^1.2pc/[l]\ar_-{i^*}@/_1.6pc/[l]
&\K^{b}(\projC)\ar^-{j_*}@/^1.2pc/[l]\ar_-{j_!}@/_1.6pc/[l]} ~~~~~~(*)$$
and
$$\xymatrix{\D^{b}(\smodC)\ar^-{j_*}[r]&\D^{b}(\smodB)\ar^-{i^!}[r]
\ar^-{j^\#}@/^1.2pc/[l]\ar_-{j^{!}}@/_1.6pc/[l]
&\D^{b}(\smodA)\ar^-{i_\#}@/^1.2pc/[l]\ar_-{i_{*}}@/_1.6pc/[l]}~~~~~~~~(**)$$
where the functors $j^!, j_*, i_*$ and $i^!$ of the recollement $(*)$ are the restrictions
of the corresponding functors in $(**)$. Then the gluing of co-$t$-structures in $(*)$ is compatible with the gluing of $t$-structures in $(**)$.}
\end{rem}}

\medskip
{\bf {In the rest of the article}, we will focus on the first kind of gluing technique.}

{At the end of this section}, we give an example to show that not all middle simple-minded collections are of glued type. Firstly, let's present a necessary condition for a simple-minded collection being of glued type.

\begin{lem}\label{lem:sufficient condition}
Suppose $\X$ and $\Y$ are not zero categories.
 Let $S=\{T_1, \cdots, T_n\}$ be a glued simple-minded collection of $\T$. Then there exists $T_i$ such that $j^!(T_i)=0$.
\end{lem}
\begin{proof}
Since $\X$ is non-zero, then there is a non-empty subset $S$ in which each object is the image of an object in a simple-minded collection of $\X$. Because $j^!i_*=0$. Then the statement holds.
\end{proof}

\begin{exam}\label{exam:not glued}
  Let $Q$ be a quiver
  $$1~\bullet\longrightarrow \bullet~2$$
  and $A=kQ$. $A$ is a finite-dimensional hereditary algebra and the Auslander-Reiten quiver of $\D^b(\smodA)$  is
  $$\xymatrix@!0{
  &S_1[-1]\ar[dr]
  &
  &P_1\ar[dr]
  &
  &S_2[1]\ar[dr]
  &
  &S_1[1]\ar[dr]
  &
  \\
  \cdots\ar[ur]
  &
  &S_2\ar[ur]
  &
  &S_1\ar[ur]
  &
  &P_1[1]\ar[ur]
  &
  &\cdots
  }$$

$\D^b(\smodA)$ has the following decomposition with respect to $e=e_1$
$$\xymatrix{\D^b(\smodAeAe)\ar^-{i_*=i_!}[r]&\D^b(\smodA)\ar^-{j^!=j^*}[r]
\ar^-{i^!}@/^1.2pc/[l]\ar_-{i^*}@/_1.6pc/[l]
&\D^b(\smodeAe)\ar^-{j_*}@/^1.2pc/[l]\ar_-{j_!}@/_1.6pc/[l]}.$$

By an easy check, $\{P_1[1], S_1\}$ is a simple-minded collection of $\D^b(\smodA)$, but $$j^!(P_1[1])\simeq S_1[1]~\text{and}~j^!(S_1)\simeq S_1,$$
then from Lemma~\ref{lem:sufficient condition}, it is not of glued type.
\end{exam}

\section{Partial order and mutations}\label{sec:order}
In this section, let $\T$ be a triangulated category. Assume $\T$ is a recollement of triangulated categories $\X$ and $\Y$
$$\xymatrix{\X\ar^-{i_*=i_!}[rr]&&\mathscr{T}\ar^-{j^!=j^*}[rr]
\ar^-{i^!}@/^1.2pc/[ll]\ar_-{i^*}@/_1.6pc/[ll]
&&\Y\ar^-{j_*}@/^1.2pc/[ll]\ar_-{j_!}@/_1.6pc/[ll]}.$$

{Let us recall the definitions of the partial order and the mutation of simple-minded collections.} Let $S=\{S_1, \cdots, S_m\}$ and $S'=\{S'_1, \cdots, S'_m\}$ be two simple-minded collections of $\T$. Define $S\geq S'$ if $\Hom(S'_i, S_j[s])=0, s<0, 1\leq i, j \leq m$. The relation $\geq$ is a partial order on the set of isomorphism classes of simple-minded collections of $\T$ (See \cite[Proposition 7.9]{KY2014}). When $S$ and $S'$ are isomorphic, we write $S=S'$ simply.
Furthermore, one can mutate a simple-minded collection.
Let's recall its definition. According to \cite[Definition 7.5]{KY2014},  the left mutation of $S$ in $i$ is defined as $\mu^+_i(S):=\{S'_1, \cdots, S'_m\}$, where
$$
S'_l=
\begin{cases}
S_i[1],~l=i\\
\cone(g_{li}), l\neq i
\end{cases}
$$
here, $g_{li}$ is a minimal left $\add(S_i)$-approximation of $S_l[-1]$, that is there is a triangle
$$(\heartsuit)\quad S_l[-1]\stackrel{g_{li}}\longrightarrow S_{li}\longrightarrow S'_l\longrightarrow S_l.$$
Indeed, the middle morphism is a right $\add(S_i)$-approximation. Dually, one can define the right mutation of $S$ in $i$, denoted by $\mu^-_i(S)$, by choosing the cocone of a minimal right $\add(S_i)$-approximation of $S_l[1]$ when $l\neq i$,  and choosing $S_i[-1]$ when $l=i$. From \cite[After Remark 7.7]{KY2014}, we know that if $S_i$ is rigid (i.e. $\Hom(S_i, S_i[1])=0$), then both the left and right mutations in $i$ are simple-minded collections. When $\T$ is a bounded derived category of a finite-dimensional algebra, a mutation is always a simple-minded collection without additional assumption \cite[Lemma 7.8]{KY2014}.

\subsection{Gluing partial order}
Let $S_\X=\{X_1, \cdots, X_n\}\geq S'_\X=\{X'_1, \cdots, X'_n\}$ be two simple-minded collections of $\X$ and
$S_\Y=\{Y_1, \cdots, Y_n\}\geq S'_\Y=\{Y'_1, \cdots, Y'_n\}$ be two simple-minded collections of $\Y$. Denote $S^1_\T, S^2_\T, S^3_\T, S^4_\T$ be the glued simple-minded collections of $\T$ of pairs $\{S_\X, S_\Y\}, \{S'_\X, S_\Y\}, \{S_\X, S'_\Y\}$ and $\{S'_\X, S'_\Y\}$, respectively. We display the notation above in the diagram
$$\xymatrix{
S_\X\ar[rr]\ar[drrr]& &S^1_\T\ar@{--}[dl]\ar@{--}[dr]& &S_\Y\ar[ll]\ar[dlll]
\\
& S^2_\T\ar@{--}[dr]& &S^3_\T\ar@{--}[dl]&
\\
S'_\X\ar[ur]\ar[rr]& &S^4_\T& &S'_\Y.\ar[ll]\ar[ul]
}$$

\begin{thm}\label{thm:preserve order}
Under the notation and setting above.
  \begin{enumerate}
    \item $S^1_\T\geq S^2_\T\geq S^4_\T$;
    \item $S^1_\T\geq S^3_\T\geq S^4_\T$.
  \end{enumerate}
\end{thm}
\begin{proof}
For convenience, set
{ 
$$S^l_\T= \begin{cases}
      \{i_*(X_1), \cdots, i_*(X_m), W^l_1, W^l_2, \cdots, W^l_n\},& ~l=1,3\\
      \{i_*(X'_1), \cdots, i_*(X'_m), W^l_1, W^l_2, \cdots, W^l_n\},& ~l=2,4
  \end{cases}$$}

(1)
Firstly, we show $S^1_\T\geq S^2_\T$. {Thanks to Lemma~\ref{lem:image of W} and the properties of $t$-structures,} this is done by the statements below.
\begin{itemize}
  \item [(i)] $\Hom(i_*(X'_i), i_*(X_j)[<0])\simeq \Hom(X'_i, X_j[<0])=0$.
  \item [(ii)] $\Hom(i_*(X'_i), W^1_j[<0])\simeq \Hom(X'_i, {i^!}(W^1_j[<0]))\simeq \Hom(X'_i, V^1_j[<0])=0$.
  \item [(iii)] $\Hom(W^2_i, i_*(X_j)[<0])\simeq \Hom(i^*(W^2_i), X_j[<0])\simeq \Hom(U^2_i[1], X_j[<0])=0$.
  \item [(iv)] By applying $\Hom(-, W^1_j[<0])$ { to} the triangle $(\#)$ w.r.t $W^2_i$, then there is an exact sequence
$$\Hom(i_*(U^2_i){[1]}, W^1_j[<0])\longrightarrow \Hom(W^2_i, W^1_j[<0])\longrightarrow \Hom(j_!(Y_i), W^1_j[<0])$$
in which
{ $$\Hom(i_*(U^2_i){[1]}, W^1_j[<0])\simeq \Hom(U^2_i[1], i^!(W^1_j)[<0])\simeq \Hom(U^2_i[1], V^1_j)=0$$}
and
$$\Hom(j_!(Y_i), W^1_j[<0])\simeq \Hom(Y_i, j^!(W^1_j)[<0])\simeq \Hom(Y_i, Y_j[<0])=0$$
{The last equality holds because $S_{\Y}$ is a simple-minded collection.} Then we have $\Hom(W^2_i, W^1_j[<0])=0$.
\end{itemize}

Next, we show $S^2_\T\geq S^4_\T$. {Since the first $m$ terms of the two simple-minded collections are the same, and by the definitions of simple-minded collection and partial order,} it suffices to prove $\Hom(W^4_i, W^2_j[<0])=0$. Let $s< 0$. By applying $\Hom(-, W^2_i[s])$ { to} the triangle $(\#)$ w.r.t. $W^4_j$,
there is an exact sequence
$$\Hom(i_*(U'^4_j[1]), W^2_i[s])\longrightarrow \Hom(W^4_j, W^2_i[s])\longrightarrow\Hom(j_!(Y'_j), W^2_i[s]).$$
Since
$$\Hom(i_*(U'^4_j[1], W^2_i[s])\simeq \Hom(U'^4_j[1], i^!(W^2_i[s]))\simeq \Hom(U'^4_j[1], V'^2_i[s])=0$$
and
$$\Hom(j_!(Y'_j), W^2_i[s])\simeq \Hom(Y'_j, j^!(W^2_i)[s])\simeq \Hom(Y'_j, Y_i[s])=0,$$
{ the last equality holds since $S_{\Y}\geq S_{\Y'}$.} Then we have $\Hom(W^4_j, W^2_i[s])=0$.

(2) It is similar to (1).
\end{proof}

\subsection{Partial order of mutations}

{The partial order and the reformulation in terms of aisles of $t$-structures are well-known. In this section, we present it in terms of simple-minded collections. For the convenience of the reader, we give detailed proofs.}

\begin{prop}\label{prop:mutation}
Let $S=\{X_1, \cdots, X_m\}$ be a simple-minded collection of $\X$. Suppose that $X_i$ is rigid. Then
$$S[-1]\geq \mu^-_i(S)\geq S\geq \mu^+_i(S)\geq S[1].$$
\end{prop}
\begin{proof}
  Set $\mu^+_i(S)=\{X'_1, \cdots, X'_m\}$ and $\mu^-_i(S)=\{X''_1, \cdots, X''_m\}$. It suffices to prove that for $1\leq l, j\leq  m, s< 0$, the following four statements hold.
\begin{enumerate}
  \item $\Hom(X'_l, X_j[s])=0$;
  \item $\Hom(X_j, X''_l[s])=0$;
  \item $\Hom(X_j[1], X'_l[s])=0$;
  \item $\Hom(X''_l, X_j[s-1])=0$.
\end{enumerate}
We only prove (1) and (3). (2) and (4) are left to the reader.

(1) Let $s< 0$. When $l=i$, then
$$\Hom(X'_i, X_j[s])=\Hom(X_i[1], X_j[s])=0.$$
When $l\neq i$, by applying $\Hom(-, X_j[s])$ { to} the triangle $(\heartsuit)$, then there is an exact sequence
$$\Hom(X_l, X_j[s])\longrightarrow\Hom(X'_l, X_j[s])\longrightarrow \Hom(X_{li}, X_j[s]).$$
Since the left and right terms are $0$, then so is the middle one. Thus (1) holds.

(3)  Let $s< 0$. When $l=i$, { $\Hom(X_j[1], X'_i[s])=\Hom(X_j[1], X_i[s+1])=\Hom(X_j, X_i[s])=0.$}
When $l\neq i$, by applying $\Hom(X_j[1], -[s])$ { to} the triangle $(\heartsuit)$, there is an exact sequence
$$\Hom(X_j[1], X_{li}[s])\longrightarrow\Hom(X_j[1], X'_l[s])\longrightarrow \Hom(X_j[1], X_l[s]).$$
Since the left and right {terms} are $0$, then so is the middle one. Thus (3) holds.
\end{proof}

From the proposition above, we have the following corollary.

\begin{cor}\label{cor:compare torsion}
The following holds.
  $$\Filt S[\geq 1]\subseteq \Filt \mu^+_i(S)[\geq 0]\subseteq \Filt S[\geq 0]\subseteq \Filt \mu^-_i(S)[\geq 0]\subseteq \Filt S[\geq -1].$$
\end{cor}

\begin{prop}
  Let $S=\{X_1, \cdots, X_m\}\geq S'=\{X'_1, \cdots, X'_m\}$ be two simple-minded collections of $\X$. Assume that $X_i, X'_j$ are rigid.
  \begin{enumerate}
    \item If $\Hom(X'_l, X_i)=0$ for $1\leq l\leq m$, then $\mu^+_i(S)\geq S'\geq \mu^+_j(S')$;
    \item If $\Hom(X'_j, X_l)=0$ for $1\leq l\leq m$, then $\mu^-_i(S)\geq S\geq \mu^-_j(S')$.
  \end{enumerate}
\end{prop}

\begin{proof}
 We prove (1). (2) is dual. Let $s<0$. Because of Proposition~\ref{prop:mutation}, it suffices to check the equations
 \begin{itemize}
   \item [(i)] $\Hom(X'_j, X_i[1][s])=0$;
   \item [(ii)] $\Hom(X'_j, \cone(g_{ri})[s])=0, r\neq i$;
 \end{itemize}

 (i) { It is direct from the assumptions.}
 (ii)
 For $r\neq i$, by applying $\Hom(X'_j, -[s])$ { to} the triangle {$(\heartsuit)$}
 $$X_r[-1]\stackrel{g_{ri}}\longrightarrow X_{ri}\longrightarrow\cone(g_{ri})\longrightarrow X_r,$$
 then there is an exact sequence
 $$ (X'_j, X_{ri}[s])\longrightarrow (X'_j, \cone(g_{ri})[s])\longrightarrow (X'_j, X_r[s])$$
 in which the left and right terms are 0, then so is the middle one.
\end{proof}

\subsection{Gluing mutations}
For convenience, in this subsection, we fix the {notation}
\begin{align*}
  S_\X & =\{X_1, \cdots, X_m\} \\
  S_\Y & =\{Y_1, \cdots, Y_n\} \\
  S_\T & =\{i_*(X_1), \cdots, i_*(X_m), W_{m+1}, \cdots, W_{m+n}\}
\end{align*}
where $S_\T$ is the glued simple-minded collection of $S_\X$ and $S_\Y$ along the recollement.

\begin{lem}\label{lem: rigidity}
The following hold.
\begin{enumerate}
  \item If $X_i$ is rigid, then so is $i_*(X_i)$;
  \item If $Y_i$ is rigid, then so are $W_i$ and $P_i$.
\end{enumerate}
\end{lem}

\begin{proof}
  The first statement is direct.
By combining the triangle $(\#)$ and the shift of the triangle $(\#\#)$ (Section~\ref{sec:gluing}), there is a diagram with exact rows and columns
$$\xymatrix{
(i_*(U_i)[1], i_*(V_i)[1])\ar[d]&  & (i_*(U_i)[1], j_*(Y_i)[1])\ar[d]
\\
(W_i, i_*(V_i)[1])\ar[r]\ar[d]&(W_i, W_i[1])\ar[r] & (W_i, j_*(Y_i)[1])\ar[d]
\\
(j_!(Y_i), i_*(V_i)[1])& &(j_!(Y_i), j_*(Y_i)[1]).
}$$
{Using the adjunctions and condition (3) of a recollement,} we have the left corners and the {upper right} corner are 0. Since
$$\Hom(j_!(Y_i), j_*(Y_i)[1])\simeq\Hom(j^!j_!(Y_i), Y_i[1])\simeq\Hom(Y_i, Y_i[1])=0,$$
{the lower right corner is 0}. Thus we have $\Hom(W_i, i_*(V_i)[1])=\Hom(W_i, j_*(Y_i)[1])=0$. This implies the center {term} $\Hom(W_i, W_i[1])=0$. By a similar proof, one can prove $P_i$ is rigid.
\end{proof}

Let $i\in\{1, \cdots, m\}$ and assume $X_i$ is rigid. Denote $S^+_l$ (resp., $S^-_l$) the simple-minded collection of $\T$ by gluing $\mu^+_i(S_\X)$ (resp., $\mu^-_i(S_\X)$) and $S_\Y$. Let $j\in\{1, \cdots, n\}$ and assume $Y_j$ is rigid. Denote $S^+_r$ (resp., $S^-_r$) the simple-minded collection of $\T$ by gluing $S_\X$ and $\mu^+_j(S_\Y)$ (resp., $\mu^-_j(S_\Y))$, here the letters { ``l" and``r"} denote the sides { ``left" and ``right".}

Based on the lemma above, under the assumption of $X_i$ or $Y_j$ being rigid, we can also do left and right mutations of the glued simple-minded collections $S_\T$ in $i$ and $m+j$. We exhibit the notation and operations in the diagrams
$$\xymatrix{
S_\X\ar[r]\ar[d]^{\mu^+_i(\mu^-_i)}
&S_\T\ar@{-->}[d]^{\mu^+_i(\mu^-_i)}
&S_\Y\ar[l]\ar@{=}[d]^{\quad \text{and}}
&S_\X\ar[r]\ar@{=}[d]
&S_\T\ar@{-->}[d]^{\mu^+_{m+j}(\mu^-_{m+j})}
&S_\Y\ar[l]\ar[d]^{\mu^+_j(\mu^-_j)}
\\
S^+_\X(S^-_\X)\ar[r]
&S^+_l(S^-_l)
&S_\Y\ar[l]
&S_\X\ar[r]
&S^+_r(S^-_r)
&S^+_\Y(S^-_\Y).\ar[l]
}$$

\begin{lem}\label{lem:mutation order}
  $S^-_*\geq S_\T\geq S^+_*$, where $*\in\{l, r\}$.
\end{lem}
\begin{proof}
  It is direct from Theorem~\ref{thm:preserve order}.
\end{proof}

\begin{lem}\label{lem:first m-items}
The following hold.
\begin{enumerate}
  \item The first $m$ { terms} of $S^+_l$ and $\mu^+_i(S_\T)$, $S^-_l$ and $\mu^-_i(S_\T)$ are the same (up to an isomorphism);
  \item The first $m$ { terms} of $S^+_r$ and $\mu^+_{m+j}(S_\T)$ are the same if and only if
      $$\Hom(i_*(X_t), W_j[1])=0, 1\leq t\leq m;$$
  \item The first $m$ { terms} of $S^-_r$ and $\mu^-_{m+j}(S_\T)$ are the same if and only if
      $$\Hom(W_j, i_*(X_t)[1])=0, 1\leq t\leq m.$$
\end{enumerate}
\end{lem}
\begin{proof}
(1) Let $1\leq t\leq m$.
If $t=i$, since $i_*(X'_t)=i_*(X_t)[1]$, then the $t$-th terms of $S^+_l$ and $\mu^+_i(S_\T)$ are the same. If $t\neq i$, then $i_*(X'_t)=i_*(\cone(g_{ti}))$ {is} the $t$-th term  in $S^+_l$. Consider the triangle
$$i_*(X_t[-1])\stackrel{i_*(g_{ti})}\longrightarrow i_*(X_{ti})\longrightarrow i_*(\cone(g_{ti}))\longrightarrow i_*(X_t).$$
It is obvious that $i_*(g_{ti})$ is a minimal left $\add i_*(X_i)$-approximation. So the $t$-th term in $\mu_i^+(S_\T)$  is $i_*(\cone(g_{ti}))$ which coincides with  the $t$-th term in $S^+_l$.

(2) Let $1\leq t\leq m$. The $t$-th term of $S^+_r$ is $i_*(X_t)$. By definition, the $t$-th term of $\mu^+_{m+j}(S_\T)$ is the cone of a minimal left $\add(W_j)$ approximation, i.e.
$$i_*(X_t)[-1]\stackrel{g_{t(m+j)}}\longrightarrow W_{jt}\longrightarrow\cone(g_{t(m+j)})\longrightarrow i_*(X_t).$$
Notice that $\cone(g_{(m+j)t})\simeq i_*(X_t)$ if and only if $g_{(m+j)t}=0$ if and only if $\Hom(i_*(X_t), W_j[1])=0$.

(3) It is a dual statement of (2).
\end{proof}

\begin{rem}\label{rem:key rem}
  We give some comments on the assumptions in the Lemma~\ref{lem:first m-items}.

   On the one hand, there are the isomorphisms (the second isomorphism in each equation is by Lemma~\ref{lem:image of W}).
$$\Hom(i_*(X_t), W_j[1])\simeq\Hom(X_t, i^!(W_j)[1])\simeq\Hom(X_t, V_j[1])$$
and
$$\Hom(W_j, i_*(X_t)[1])\simeq\Hom(i^*(W_j), X_t[1])\simeq\Hom(U_j, X_t).$$

On the other hand, by applying $\Hom(i_*(X_t), -[1])$ { to} the triangle $(\#\#)$ w.r.t. $W_j$, there is the following exact sequence
  $$\xymatrix{
  (i_*(X_t), j_!(Y_j)[1])\ar[r]&(i_*(X_t), W_j[1])\ar[r]&(i_*(X_t), i_*(U_j)[2])
  }$$
  So there is a sufficient condition to satisfy the condition in Lemma~\ref{lem:first m-items}(2):
  $$\Hom(i_*(X_t), j_!(Y_j)[1])=0=\Hom(X_t, X_l[s]), 1\leq t, l\leq m, s\geq 2.$$
  Dually, there is a sufficient condition to satisfy the condition in Lemma~\ref{lem:first m-items}(3): $$\Hom(j_*(Y_j), i_*(X_t)[1])=0=\Hom(X_t, X_l[s]), 1\leq t, l\leq m, s\geq 2.$$
\end{rem}

\begin{lem}\label{lem:order left-right mutation}
The following hold.
\begin{enumerate}
  \item $\mu^+_i(S_\T)\geq S^+_l$, and dually $S^-_l\geq \mu^-_i(S_\T)$;
  \item $\mu^+_{m+j}(S_\T)\geq S^+_r$, and dually $S^-_r\geq \mu^-_{m+j}(S_\T)$.
\end{enumerate}
\end{lem}
\begin{proof}
(1) By Lemma~\ref{lem:first m-items}, we set
\begin{align*}
  \mu^+_i(S_\X) &= \{X'_1, \cdots, X'_m\} \\
 S^+_l  & =\{i_*(X'_1), \cdots, i_*(X'_m), W'_{m+1}, \cdots, W'_{m+n}\} \\
 \mu_i^+(S_\T) & =\{i_*(X'_1), \cdots, i_*(X'_m), W''_{m+1}, \cdots, W''_{m+n}\}
\end{align*}
Let $s<0$.
It suffices to prove ${ \Hom(W'_l, W''_j[s])=0}, m+1\leq l,  j\leq m+n$.
By applying $\Hom(W'_l, -[s])$ { to} the triangle
$$W_j[-1]\longrightarrow i_*(X_{ji})\longrightarrow W''_j\longrightarrow W_j,$$
then there is an exact sequence
$$(W'_l, i_*(X_{ji})[s])\longrightarrow(W'_l, W''_j[s])\longrightarrow(W'_l, W_l[s])$$
Since $S_\T\geq S^+_l$ (Lemma~\ref{lem:mutation order}), then the left and right terms are 0, so is the middle one.

(2) Set
\begin{align*}
  \mu^+_j(S_\Y)= & \{Y'_1, \cdots, Y'_m\}, \\
   S^+_r=&\{i_*(X_1), \cdots, i_*(X_m), W'_{m+1}, \cdots, W'_{m+n}\}  \\
 \mu^+_{m+j}(S_\T)=& \{W''_1, \cdots, W''_m, W''_{m+1}, \cdots, W''_{m+n}\}
\end{align*}
where $W''_{m+j}=W_j[1]$.
Let $s<0$. The statement is equivalent to the following four claims.

{\bf Claim 1:} $\Hom(i_*(X_l), W''_{m+j}[s])=0, 1\leq l\leq m$. {This is due to adjunctions and Lemma~\ref{lem:image of W} and the properties of $t$-structures.}

{\bf Claim 2:} For $t\neq m+j$, $\Hom(i_*(X_i), W''_t[s])=0$.

When $1\leq t\leq m$, {by the rigidity of $Y_j$ and Lemma~\ref{lem: rigidity}, there is a triangle}
$$i_*(X_t)[-1]\stackrel{g_{tj}}\longrightarrow W_{tj}\longrightarrow W''_t\longrightarrow i_*(X_t).$$
By applying $\Hom(i_*(X_i), -[s])$ { to} this triangle, there is an exact sequence
$$(i_*(X_i), W_{tj}[s])\longrightarrow (i_*(X_i), W''_t[s])\longrightarrow(i_*(X_i),  i_*(X_t)[s])$$
{ in which, the left term is 0 due to adjunction and Lemma~\ref{lem:image of W}, and the right term is 0 because $S_{\X}$ is a simple-minded collection.}  Then we have the middle one is 0.

When $m+1\leq t\leq m+n$, due to the rigidity of $Y_j$ and Lemma~\ref{lem: rigidity}, there is a triangle
$$W_t[-1]\stackrel{g_{tj}}\longrightarrow W_{tj}\longrightarrow W''_t\longrightarrow W_t.$$
By applying $\Hom(i_*(X_i), -[s])$ { to} this triangle, there is an exact sequence
$$(i_*(X_i), W_{tj}[s])\longrightarrow (i_*(X_i), W''_t[s])\longrightarrow(i_*(X_i),  W_t[s]),$$
in which, the left and right terms are 0, then so is the middle one.

{\bf Claim 3:} $\Hom(W'_l, W''_{m+j}[s])=0$.

When $s<-1$, then $\Hom(W'_l, W_{m+j}[1][s])=0$ by the fact $S_\T\geq S^+_r$ (Lemma~\ref{lem:mutation order}).

When $s=-1$, by applying $\Hom(-, W_{m+j})$ { to} the triangle
$$i_*(U'_l)\longrightarrow j_!(Y'_l)\longrightarrow W'_l\longrightarrow i_*(U'_l)[1],$$
there is an exact sequence
$$(i_*(U'_l)[1], W_{m+j})\longrightarrow (W'_l, W_{m+j})\longrightarrow (j_!(Y'_l), W_{m+j})$$
in which $\Hom(i_*(U'_l)[1], W_{m+j})\simeq \Hom(U'_l[1], i^!(W_{m+j}))\simeq \Hom(U'_l[1], V_j)=0$ and
$\Hom(j_!(Y'_l), W_{m+j})\simeq \Hom(Y'_l, j^!(W_{m+j}))\simeq \Hom(Y'_l, Y_j)=\Hom(Y'_l, Y'_j[-1])=0.$
Thus, the middle term $\Hom(W'_l, W_{m+j})=0$.

{\bf Claim 4:} For $t\neq m+j$, $\Hom(W'_l, W''_t[s])=0$. By a similar way as the proof of Claim 2 and use the fact $S_\T\geq S^+_r$ (Lemma~\ref{lem:mutation order}).
\end{proof}

The following lemma plays a key role in the proof of the main theorem below. Let
$$g_{t(m+j)}: W_t[-1]\longrightarrow W_{tj}$$
be a minimal left $\add (W_{m+j})$-approximation.

\begin{lem}\label{lem:key lem}
 With the condition in Lemma~\ref{lem:first m-items}(2), i.e. 
$${\Hom(i_*(X_t), W_j[1])=0, 1\leq t\leq m.}$$ 
Then the morphism
    $$j^!(g_{t(m+j)}): j^!(W_t[-1])\longrightarrow j^!(W_{tj})$$
is a minimal left $\add(Y_j)$-approximation. Moreover, there exists the commutative diagram of triangles
$$\xymatrix{
Y_t[-1]\ar[d]_{\simeq}\ar[r]^{g_{tj}}
&Y_{tj}\ar[d]^{\simeq}\ar[r]
&Y'_t\ar[d]^{\simeq}
\\
 j^!(W_t[-1])\ar[r]^{j^!(g_{t(m+j)})}
 &j^!(W_{tj})\ar[r]
 &j^!(W'_t)
}$$
\end{lem}
\begin{proof}
  There are commutative diagram
  $$\xymatrix{
  \Hom(j^!(W_{tj}),Y_j)\ar[d]_{j^!(g_{t(m+j)})^*}\ar[rr]^{\stackrel{\text{{adjunction}}}{\simeq}}
  & &\Hom(W_{tj}, j_*(Y_j))\ar[d]^{g_{t(m+j)}^*}
  \\
   \Hom( j^!(W_t[-1]), Y_j)\ar[rr]^{\stackrel{{{\text{adjunction}}}}{\simeq}}
  & &\Hom(W_t[-1], j_*(Y_j))
  }$$
  and commutative diagram {that is obtained by applying the functors $\Hom_{\T}(W_{tj},-)$ and $\Hom_{\T}(W_{t}[-1],-)$ to the triangle $(\#\#)$}
  {\xiaowuhao
  $$\xymatrix{
  (W_{tj}, i_*(V_j))\ar[r]\ar[d]
  &(W_{tj}, W_j)\ar[r]\ar[d]^{(g_{t(m+j)}, W_j)}
  &(W_{tj}, j_*(Y_j))\ar[r]\ar[d]^{g_{t(m+j)}^*}
  &(W_{tj}, i_*(V_j)[1])\ar[d]
  \\
  (W_t[-1], i_*(V_j))\ar[r]
  &(W_t[-1], W_j)\ar[r]
  &(W_t[-1], j_*(Y_j))\ar[r]
  &(W_t[-1], i_*(V_j)[1])
  .}$$}
 In which
 $$\Hom(W_j, i_*(V_j)[k])\simeq\Hom(i^*(W_j), V_j[k])\simeq\Hom(U_j[1], V_j[k])=0, k=0,1$$
 and { by Lemma~\ref{lem:image of W}} 
$${\Hom(W_t[-1], i_*(V_j))\simeq\Hom(i^*(W_t[-1]), V_j)\simeq \Hom(U_t, V_j)=0.}$$
Note that $\Hom(W_t[-1], i_*(V_j)[1])\simeq\Hom(U_t, V_j[1])$ and thanks to the assumption, we have $\Hom(X_t, V_j[1])=0$ (see Remark~\ref{rem:key rem}), this implies $\Hom(U_t, V_j[1])=0$. Then
  the four corners in the diagram above are 0, and so the middle horizontal two morphisms are {isomorphisms}. Because $g_{t(m+j)}$ is a $\add(W_j)$-approximation, then $\Hom(g_{t(m+j)}, W_j)$ is surjective, and so is $g_{t(m+j)}^*$. Therefore $j^!(g_{t(m+j)})^*$ in the first diagram is surjective, since $j^!(W_{tj})\in\add(j^!(W_{m+j}))=\add(Y_j)$, so $j^!(g_{t(m+j)})$ is a left $\add(Y_j)$-approximation.

  Next, we show this approximation is left minimal. Note that we have the commutative diagram {that is obtained from the unit of the adjunction}
  $$\xymatrix{
  W_t[-1]\ar[r]^{\eta_t}\ar[d]_{g_{tj}}
  &j_*j^!(W_t[-1])\ar[d]^{j_*j^!(g_{tj})}
  \\
  W_{tj}\ar[r]^{\eta_{tj}}
  &j_*j^!(W_{tj})
  }$$
To show $j^!(g_{tj})$ is left minimal, {it is equivalent to }show $j_*j^!(g_{tj})$ is left minimal. And moreover, it suffices to show $\xi=j_*j^!(g_{tj})\eta_t$ is left minimal.

Let $\alpha: j_*j^!(W_{tj})\to j_*j^!(W_{tj})$ satisfy $\alpha \xi= \xi$. {Note that we have the diagram below,} since
$$\Hom(W_{tj}, i_*i^!(W_{tj})[1])\simeq\Hom(i^*W_{tj}, i^!(W_{tj})[1])\simeq\Hom(U_{tj}, V_{tj})=0,$$ 
there exits $\beta$ such that the middle square is commutative,
$$\xymatrix{
&&W_t[-1]\ar[d]^\xi\ar[dl]_{g_{tj}}&
\\
&W_{tj}\ar[r]^{\eta_{tj}}\ar@{-->}[d]^{\beta}&j_*j^!(W_{tj})\ar[d]^{\alpha}&
\\
i_*i^!(W_{tj})\ar[r]&W_{tj}\ar[r]^{\eta_{tj}}&j_*j^!(W_{tj})\ar[r]&i_*i^!(W_{tj})[1]
}$$
Since $\Hom(W_t[-1], i_*i^!(W_{tj}))\simeq\Hom(U_t, V_j[1])=0$, then the map
$$\Hom(W_t[-1], W_{tj})\stackrel{\eta_{tj}^*}{\longrightarrow}\Hom(W_t[-1], j_*j^!(W_{tj}))$$
is injective. Because $\eta_{tj}^*(\beta g_{tj})=\eta_{tj}\beta g_{tj}=\alpha \xi=\xi=\eta_{tj}^*(g_{tj})$, then $\beta g_{tj}=g_{tj}$. Since $g_{tj}$ is left minimal, then $\beta$ is an isomorphism, and so is $\alpha$. In fact, in this case, $\alpha= J_*j^!(\beta)$ Thus $\xi$ is left minimal.

Because $j^!(W_t[-1])\simeq Y_t[-1]$ and $g_{tj}: Y_t[-1]\to Y_{tj}$ is left minimal, then there exists commutative diagram
$$\xymatrix{
Y_t[-1]\ar[d]_{\simeq}\ar[r]^{g_{tj}}
&Y_{tj}\ar[d]^{\simeq}
\\
 j^!(W_t[-1])\ar[r]^{j^!(g_{t(m+j)})}
 &j^!(W_{tj})
}$$
which can {be extended} to an isomorphism of triangles.
\end{proof}

\begin{thm}\label{thm:mutation smc}
With the notation and assumption above.
\begin{enumerate}
  \item $S^+_l=\mu_i^+(S_\T)$, and dually $S^-_l=\mu_i^-(S_\T)$, where $i\in \{1,\cdots, m\}$;
  \item For $j\in \{m+1,\cdots, m+n\}$, if $\Hom(i_*(X_t), W_{j}[1])=0, 1\leq t\leq m$,  then $S^+_r=\mu^+_{m+j}(S_\T)$;
  \item For $j\in \{m+1,\cdots, m+n\}$, if $\Hom(W_{j}, i_*(X_t)[1])=0, 1\leq t\leq m$,  then $S^-_r=\mu^-_{m+j}(S_\T)$.
\end{enumerate}
\end{thm}
\begin{proof}
We only prove (1) and (2). (3) is a dual statement of (2) {and can be proved by using the dual statement of Lemma~\ref{lem:key lem}.}

(1)  Thanks to Lemma~\ref{lem:first m-items}, set
\begin{align*}
  \mu^+_i(S_\X) &= \{X'_1, \cdots, X'_m\} \\
 S^+_l  & =\{i_*(X'_1), \cdots, i_*(X'_m), W'_{m+1}, \cdots, W'_{m+n}\} \\
 \mu_i^+(S_\T) & =\{i_*(X'_1), \cdots, i_*(X'_m), W''_{m+1}, \cdots, W''_{m+n}\}
\end{align*}

Because we have Lemma~\ref{lem:order left-right mutation}, so it suffices to
 show that $S^+_l\geq \mu_i^+(S_\T)$. It {is equivalent to} prove that for $s<0$,
 $$\Hom(W''_j, W'_l[s])=0, m+1\leq j, l\leq m+n.$$

 By applying $\Hom(-, W'_l[s])$ { to} the triangle
  $$W_j[-1]\longrightarrow i_*(X_{ji})\longrightarrow W''_j\longrightarrow W_j,$$
  then there is an exact sequence
$$(i_*(X_{ji})[1], W'_l[s])\longrightarrow(W_j,  W'_l[s])\longrightarrow(W''_j,  W'_l[s]))\longrightarrow (i_*(X_{ji}), W'_l[s])$$
Since $X_{ji}\in\add(X_i)$ and $\Hom(i_*(X_i)[k], W'_l[s])\simeq \Hom(i_*(X'_i)[k-1], W'_l[s])=0, k=0,1$, { the last equality holds because $S^+_l$ is a simple-minded collection,} then
$$\Hom(W_j,  W'_l[s])\simeq \Hom(W''_j,  W'_l[s]).$$
By applying $\Hom(W_j, -[s])$ { to} the triangle $(\#\#)$ with respect to $W'_l$ (Section~\ref{sec:gluing}), then there is an exact sequence
$$\Hom(W_j, i_*(V'_l)[s])\longrightarrow \Hom(W_j,  W'_l[s])\longrightarrow \Hom(W_j, j_*(Y_l)[s]).$$
Since $U_j[1]\in\Filt S_\X[\geq 1]\subset\Filt S_\X'[\geq 0]$ by Corollary~\ref{cor:compare torsion} and $V'_l[s]\in\Filt S_\X'[\leq -1]$, then
$$\Hom(W_j,  i_*(V'_l)[s])\simeq \Hom(i^*(W_j), V'_l[s])\simeq\Hom(U_j[1], V'_l[s])=0.$$
Furthermore, we have
$$\Hom(W_j, j_*(Y_l)[s])\simeq \Hom(j^!(W_j), Y_l[s])\simeq \Hom(Y_j, Y_l[s])=0.$$
Thus $\Hom(W_j,  W'_l[s])=0$, and this implies that $\Hom(W''_j, W'_l[s])=0$. The dual statement is proved similarly.

(2) {The strategy of the proof is the same as that for part (1) by using the partial order on simple-minded collections.} Under our assumption and by Lemma~\ref{lem:first m-items}, set
\begin{align*}
  \mu^+_j(S_\Y)= & \{Y'_1, \cdots, Y'_m\}, \\
   S^+_r=&\{i_*(X_1), \cdots, i_*(X_m), W'_{m+1}, \cdots, W'_{m+n}\}  \\
 \mu^+_{j}(S_\T)=& \{i_*(X_1), \cdots, i_*(X_m), W''_{m+1}, \cdots, W''_{m+n}\}
\end{align*}
where $W''_{m+j}=W_{m+j}[1]$.  { By Lemma~\ref{lem:order left-right mutation}(2), we need to prove $S_r^+\geq \mu^+_{m+j}(S_\T)$.} It is equivalent to prove that for $s<0$, the following equations hold.
\begin{itemize}
  \item [(i)] $\Hom(W_{m+j}[1], W'_l[s])=0, m+1\leq l\leq m+n$;
  \item [(ii)] $\Hom(W''_t, W'_l[s])=0, m+1\leq t(\neq m+j), l(\neq m+j)\leq m+n$;
  \item [(iii)] $\Hom(W''_t, W'_{m+j}[s])=0, m+1\leq t(\neq m+j)\leq m+n.$
\end{itemize}

(i)
By applying $\Hom(W_{m+j}[1], -[s])$ { to} the triangle
$$i_*(V'_l)\longrightarrow W'_l\longrightarrow j_*(Y'_l)\longrightarrow i_*(V'_l)[1],$$
there is an exact sequence
$$(W_{m+j}, i_*(V'_l)[s-1])\longrightarrow (W_{m+j}, W'_l[s-1])\longrightarrow(W_{m+j}, j_*(Y'_l)[s-1])$$
in which
$$(W_{m+j}, i_*(V'_l)[s-1])\simeq (i^*(W_{m+j}), V'_l[s-1])\simeq (U_i, V'_l[s-2])=0,$$
{ the last equality holds because $U_i\in \Filt S_\X[\geq 0], V'_i\in \Filt S_\X[\leq 0]$}. And
 $$(W_{m+j}, j_*(Y'_l)[s-1])\simeq (j^!(W_{m+j}), Y'_l[s-1])\simeq(Y_j, Y'_l[s-1])=0,$$
{ the last equality holds since $\mu^+_j(S_\Y)\geq S_\Y[1]$ (see Proposition~\ref{prop:mutation})}. Therefore we have $\Hom(W_{m+j}[1], W'_l[s])=0$.

(ii)
By applying $\Hom(W''_t, -[s])$ { to the shift of triangle $(\#\#)$}
$$j_*(Y'_l)[-1]\longrightarrow i_*(V'_l)\longrightarrow W'_l\longrightarrow j_*(Y'_l),$$
there is an exact sequence
$$(\spadesuit)\quad (W''_t, i_*(V'_l)[s])\longrightarrow(W''_t, W'_l[s])\longrightarrow(W''_t, j_*(Y'_l)[s]).$$
By applying $\Hom(-, i_*(V'_l)[s])$ { to} the triangle in { $(\heartsuit)$}
$$W_t[-1]\longrightarrow W_{tj}\longrightarrow W''_t\longrightarrow W_t.$$
there is an exact sequence
$$(W_t, i_*(V'_l)[s])\longrightarrow(W''_t, i_*(V'_l)[s])\longrightarrow(W_{tj}, i_*(V'_l)[s])$$
in which
$$\Hom(W_t, i_*(V'_l)[s])\simeq \Hom(i^*(W_{tj}), V'_l[s])\simeq \Hom(U_{tj}[1], V'_l[s])=0$$
and
$$\Hom(W_{tj}, i_*(V'_l)[s])\simeq \Hom(i^*(W_t), V'_l[s])\simeq \Hom(U_t[1], V'_l[s])=0.$$
Thus the middle term $\Hom(W''_t, i_*(V'_l)[s])=0$. Moreover, when $l\neq j$, by applying $\Hom(W''_t, j_*(-)[s])$ { to} the triangle
$$Y_l[-1]\longrightarrow Y_{lj}\longrightarrow Y'_l\longrightarrow Y_l,$$
then there is a diagram with exact rows and columns
$$\xymatrix{
(W_t, j_*(Y_{lj})[s])\ar[d]& &(W_t, j_*(Y_l)[s])\ar[d]
\\
(W''_t, j_*(Y_{lj})[s])\ar[r]\ar[d]&(W''_t, j_*(Y'_l)[s])\ar[r]&(W''_t, j_*(Y_l)[s])\ar[d]
\\
(W_{tj}, j_*(Y_{lj})[s])& &(W_{tj}, j_*(Y_l)[s])
}$$
{ Using adjunctions and Lemma~\ref{lem:image of W}, the four corners are 0,} so is the central one. Thus we have shown the left and right terms in the sequence $(\spadesuit)$ are 0, then so is the middle one.

(iii)
By applying $\Hom(-,  W'_{m+j}[s])$ { to} the triangle
$$W_t[-1]\stackrel{g_{tj}}\longrightarrow W_{tj}\longrightarrow W''_t\longrightarrow W_t,$$
there is an exact sequence
$$(W_t, W'_{m+j}[s])\to (W''_t, W'_{m+j}[s])\to (W_{tj}, W'_{m+j}[s])\stackrel{g_{tj}^*}\to (W_t[-1], W'_{m+j}[s]).$$

To compute $\Hom(W_t, W'_{m+j}[s])$, act $\Hom(W_t, -[s])$ { to} the triangle $(\#\#)$ w.r.t. $W'_{m+j}$ (Section~\ref{sec:gluing}), then there is an exact sequence
$$(W_t, i_*(V'_j)[s])\longrightarrow  (W_t, W'_{m+j}[s])\longrightarrow  (W_t, j_*(Y'_j)[s]) $$
in which
$$(W_t, i_*(V'_j)[s])\simeq (i^*(W_t), V'_j[s])\simeq (U_t[1], V'_j[s])=0$$
and
$$(W_t, j_*(Y'_j)[s])\simeq(j^!(W_t), Y'_j[s])\simeq (Y_t, Y_j[1+s])=0,$$
{ the last equality holds because $t\neq j$ and $S_{\Y}$ is a simple-minded collection.} Thus we have $\Hom(W_t, W'_{m+j}[s])=0$.

Next, we prove that $g_{tj}^*$ is injective. Note that the triangle ($\#\#$)
$$ i_*(V'_j)[s]\longrightarrow W'_{m+j}[s]\longrightarrow j_*(Y'_j)[s]\longrightarrow i_*(V'_j)[s+1]$$
induces the commutative diagram with exact rows
$${\xiaoliuhao
\xymatrix@C=12pt{
(W_{tj}, i_*(V'_j)[s])\ar[r]\ar[d]&(W_{tj}, W'_{m+j}[s])\ar[r]\ar[d]^{g_{tj}^*}&(W_{tj},  j_*(Y'_j)[s])\ar[r]\ar[d]^{(g_{tj}, j_*(Y'_j)[s])}&(W_{tj},  i_*(V'_j)[s+1])\ar[d]
\\
(W_t[-1], i_*(V'_j)[s])\ar[r]&(W_t[-1], W'_{m+j}[s])\ar[r]&(W_t[-1],  j_*(Y'_j)[s])\ar[r]&(W_t[-1],  i_*(V'_j)[s+1])
}}$$
Since the four corners are 0 (note that $i^*(W_t)\simeq U_t[1]$), it suffices to prove the third vertical map $\Hom(g_{tj}, j_*(Y'_j)[s])$ is injective. This is true since we have the commutative diagram
$$\xymatrix{
&(j^!(W'_t), Y'_j[s])\ar[d]
\\
(W_{tj},  j_*(Y'_j)[s])\ar[d]^{(g_{tj}, j_*(Y'_j)[s])}\ar[r]^{\stackrel{{\text{adjunction}}}{\simeq}}&(j^!(W_{tj}), Y'_j[s])\ar[d]^{(j^!(g_{tj}), Y'_j[s])}
\\
(W_t[-1],  j_*(Y'_j)[s])\ar[r]^{\stackrel{{\text{adjunction}}}{\simeq}}&(j^!(W_t)[-1], Y'_j[s]).
}$$
Thank to Lemma~\ref{lem:key lem}, we have $j^!(W'_t)\simeq Y'_t$, and then there is $\Hom(j^!(W'_t), Y'_j[s])\simeq\Hom(Y'_t, Y'_j[s])=0$ since $t\not = j$ and $Y'_t, Y'_j$ in the simple-minded collection $\mu^+_j(S_\Y)$, hence $\Hom(g_{tj}, j_*(Y'_j)[s])$ is injective, and so is $g_{tj}^*$. Therefore $ \Hom(W''_t, W'_{m+j}[s])=0$. This finishes the proof.
\end{proof}

\begin{exam}
With the notation in Example~\ref{exam:not glued}, we give the following examples to explain the main theorem.

(1) The following examples are of simple-minded collections whose mutations {\color{black}on the left-hand side of the recollement} commute with gluing.
$$\xymatrix{
\{S_2\}\ar[r]\ar[d]^{\mu^+}
&\{S_2, S_1\}\ar[d]^{\mu^+_1}
&\{S_1\}\ar[l]\ar@{=}[d]
\\
\{S_2[1]\}\ar[r]
&\{S_2[1], P_1\}
&\{S_1\}\ar[l]
}$$
and
$$\xymatrix{
\{S_2[1]\}\ar[r]\ar[d]^{\mu^-}
&\{S_2[1], P_1[-1]\}\ar[d]^{\mu^-_1}
&\{S_1[-1]\}\ar[l]\ar@{=}[d]
\\
\{S_2\}\ar[r]
&\{S_2, P_1[-1]\}
&\{S_1[-1]\}.\ar[l]
}$$

(2) The following examples are of simple-minded collections whose mutations {\color{black}on the right-hand side of the recollement} commute with gluing. One can check that the simple-minded collections involved satisfying the conditions in Theorem~\ref{thm:mutation smc}(2), (3).
$$\xymatrix{
\{S_2[1]\}\ar[r]\ar@{=}[d]
&\{S_2[1], S_1[1]\}\ar[d]^{\mu^+_2}
&\{S_1[1]\}\ar[l]\ar[d]^{\mu^+}
\\
\{S_2[1]\}\ar[r]
&\{S_2[1],S_1[2]\}
&\{S_1[2]\}\ar[l]
}$$
and
$$\xymatrix{
\{S_2[1]\}\ar[r]\ar@{=}[d]
&\{S_2[1], P_1\}\ar[d]^{\mu^-_2}
&\{S_1\}\ar[l]\ar[d]^{\mu^-}
\\
\{S_2[1]\}\ar[r]
&\{S_2[1], P_1[-1]\}
&\{S_1[-1]\}.\ar[l]
}$$

(3) The following examples are of simple-minded collections whose mutations {\color{black}on the right-hand side of the recollement} don’t commute with gluing. One can check that the simple-minded collections involved don't satisfy the conditions in Theorem~\ref{thm:mutation smc}(2), (3).
$$\xymatrix{
\{S_2[1]\}\ar[r]\ar@{=}[dd]
&\{S_2[1], P_1\}\ar[d]^{\mu^+_2}
&\{S_1\}\ar[l]\ar[dd]^{\mu^+}
\\
&\{S_1, P_1[1]\}\ar[d]^{\neq}
&
\\
\{S_2[1]\}\ar[r]\ar@{=}[dd]
&\{S_2[1], S_1[1]\}\ar[d]^{\mu^-_2}
&\{S_1[1]\}\ar[l]\ar[dd]^{\mu^-}
\\
&\{P_1[1], S_1\}\ar[d]^{\neq}
&
\\
\{S_2[1]\}\ar[r]
&\{S_2[1], P_1\}
&\{S_1\}.\ar[l]
}$$
\end{exam}


\end{document}